\documentclass[10pt,a4paper]{article}
\usepackage[final]{optional}

\usepackage{amsfonts, amsmath, wasysym}
\usepackage{tikz}
\usepackage{amssymb,amsthm,
paralist
}
\usepackage{
latexsym,
}
\usepackage{url}

\definecolor{darkgreen}{rgb}{0,0.5,0}
\definecolor{darkred}{rgb}{0.7,0,0}
\usepackage[colorlinks, 
citecolor=darkgreen, linkcolor=darkred
]{hyperref}

\theoremstyle{plain}

\numberwithin{equation}{section}

\newcommand{\h}{\ensuremath{{\mathcal H}}}

\newcommand{\pl}[2]{{\frac{\partial #1}{\partial #2}}}
\newcommand{\ti}{\tilde}
\newcommand{\al}{\alpha}

\newcommand{\Om}{\Omega}

\newcommand{\vph}{\varphi}
\newcommand{\ep}{\varepsilon}

\newcommand{\R}{\ensuremath{{\mathbb R}}}

\newcommand{\downto}{\downarrow}

\newcommand{\lap}{\Delta}

\DeclareMathOperator{\Vol}{Vol}
\DeclareMathOperator{\VolB}{VolB}

\DeclareMathOperator{\inj}{inj}
\newcommand{\beq}{\begin{equation}}
\newcommand{\eeq}{\end{equation}}
\newcommand{\beqa}{\begin{equation}\begin{aligned}}
\newcommand{\eeqa}{\end{aligned}\end{equation}}
\newcommand{\brmk}{\begin{rmk}}
\newcommand{\ermk}{\end{rmk}}
\newcommand{\partref}[1]{\hbox{(\csname @roman\endcsname{\ref{#1}})}}
\newcommand{\half}{\frac{1}{2}}

\newcommand{\Rm}{{\mathrm{Rm}}}
\newcommand{\Ric}{{\mathrm{Ric}}}

\newcommand{\K}{{\ensuremath{\mathrm{K_{IC_1}}}}}
\newcommand{\I}{{\ensuremath{\mathrm{IC_1}}}}

 \newtheorem{thm}{Theorem}[section]

\newtheorem{lem}[thm]{Lemma}

\newtheorem{rem}[thm]{Remark}

\title{\sc time zero regularity of ricci flow}
\author{Man-Chun Lee and Peter M. Topping}
\date{14 October 2022}

\begin{document}

%\author{Man-Chun Lee}
%\address[Man-Chun Lee]{Department of Mathematics, The Chinese University of Hong Kong, Shatin, N.T., Hong Kong
%}
%\email{mclee@math.cuhk.edu.hk}
%
%\author{Peter M. Topping}
%\address[Peter M. Topping]{Mathematics Institute, Zeeman Building, University of Warwick, Coventry CV4 7AL}
% \email{P.M.Topping@warwick.ac.uk}

\maketitle

\begin{abstract}
We consider the problem of when
a smooth Ricci flow, for positive time, that attains smooth initial data in a weak sense must be smooth down to the initial time.
We obtain curvature estimates for an example where this fails that was given in \cite{TY3}. We prove a positive result in the case that the flow satisfies a lower $\K$ curvature bound, equivalent to a lower Ricci bound in three dimensions. As an application, we prove that Gromov-Hausdorff limits of WPIC1 manifolds are WPIC1.
\end{abstract}

%\cmt{Summary of changes for v2:
%\begin{enumerate}
%\item
%Changed statement of Theorem \ref{WPIC1_lim_thm} a little and the first and third paragraphs following it. Changed the proof accordingly (first two paragraphs replaced by new first paragraph).
%\item
%First paragraph of proof of Theorem \ref{sym_thm} changed.
%\item
%Remark \ref{r22} adjusted.
%\item
%A few additional lines added in the proof of Theorem \ref{IC1_thm} concerning volume, plus reference.
%\item
%A few additional lines added in the proof of Theorem \ref{IC1_thm} concerning distances, plus reference.
%\item
%A few additional lines added in the proof of Theorem \ref{IC1_thm} concerning compactness, plus reference.
%\item
%Abstract changed.
%\item
%Updated references. % \cite{DSS}.
%\item
%Made explicit two Myers-Steenrod applications.
%\end{enumerate}
%}

\section{Introduction}

Suppose that $g(t)$ is a smooth complete Ricci flow on a manifold $M$ over a time interval 
$(0,T)$, that attains a smooth Riemannian metric $g_0$ on $M$ as initial data in the weak sense that 
\beq
\label{loc_unif_conv}
d_{g(t)}\to d_{g_0}\qquad\text{locally uniformly in }M\times M\text{ as }t\downto 0,
\eeq
where $d_g$ is the Riemannian distance on $M$ with respect to a metric $g$.
In this note we consider the question of when we can deduce that the extension of $g(t)$ to $[0,T)$ arising by setting $g(0):=g_0$ is smooth down to $t=0$.

A solution to this problem was found in \cite{TY3} when it was proved that there exists a smooth complete conformal  Ricci flow $g(t)$ on $\R^2$, for $t>0$, that starts with the Euclidean metric in the sense above, but which is not the static solution that remains as Euclidean space for all time. In other words, if $g_0$ is the standard Euclidean metric on $\R^2$ then \eqref{loc_unif_conv} holds for this non-static Ricci flow $g(t)$.
This provides an example of nonuniqueness of solutions of the Ricci flow in this setting.

This then also solves the following more general local version of the problem.
Consider a potentially incomplete Ricci flow $g(t)$ over a time interval $(0,T)$
on an $n$-dimensional smooth manifold $M$,  with the property that 
for some open set $\Om\subset M$, we have 
\beq
\label{loc_unif_conv2}
d_{g(t)}\to d_{0}\ \text{ uniformly in }\Om\times\Om \text{ as }t\downto 0,
\text{ for some $n$-Riemannian  metric }d_0\text{ on }\Om.
\eeq
Here, $d_0$ being $n$-Riemannian means that it is a distance metric such that every $x_0\in\Om$ has some neighbourhood in $(\Om,d_0)$ that is isometric to a smooth $n$-dimensional Riemannian manifold.
One then asks whether the flow extends smoothly to $t=0$ 
on $\Om$.

This question arises naturally when we construct a Ricci flow starting with initial data that admits singularities. Then the flow can do no more than attain the initial data in some weak sense. For example, one can consider Ricci solitons starting with a Riemannian cone, admitting the initial data in a metric sense, and consider whether the initial data is in fact attained smoothly away from the vertex of the cone.

This discussion begs the question of what additional hypotheses are required in order to ensure that the questions above can be answered in the affirmative. 
Perhaps the first result of this form is due to T. Richard \cite{TR}. His work handles 
the case of compact two-dimensional $M$ under the additional assumption that 
the Gauss curvature of the flow $g(t)$ is uniformly bounded below as $t\downto 0$.

\begin{thm}[{Implied by \cite[Proposition 3.1.4]{TR}}]
\label{TR_thm}
Suppose $(M,g_0)$ is a closed Riemannian surface and $g(t)$ is a smooth Ricci flow on $M$ for $t\in (0,T)$ with curvature uniformly bounded below, and with the property that \eqref{loc_unif_conv} holds. Then if we define $g(0):=g_0$ we obtain an extended Ricci flow that is smooth on $[0,T)$.
\end{thm}

Without the assumption that the curvature is bounded below the result fails, as seen by a minor modification of the example of the second author and Yin given in \cite{TY3}.

More recently, the question in general dimension was considered by A. Deruelle, F. Schulze and M. Simon \cite{DSS}. They proved that if one assumes both a uniform lower Ricci curvature bound (analogous to the lower Gauss curvature bound of Theorem \ref{TR_thm}) and an upper curvature bound of the form $|\Rm|_{g(t)}\leq C/t$, then one can deduce time zero regularity.

\begin{thm}[{Implied by \cite[Theorem 1.6]{DSS}}]
\label{DSS_thm}
Suppose 
$g(t)$ is a smooth (possibly incomplete) Ricci flow on $M$ for $t\in (0,T)$ satisfying 
for some $C<\infty$ that 
\begin{compactenum}
\item
\label{cond_1}
$\Ric\geq  -C$ and 
\item
\label{cond_2}
$|\Rm|_{g(t)}\leq C/t$. % for all $t\in (0,T)$,
\end{compactenum}
and with the property that \eqref{loc_unif_conv2} holds for some $\Om\subset M$. 
Then on $\Om$ we can extend $g(t)$ to a smooth Ricci flow on the time interval 
$[0,T)$.
\end{thm}

The theorem above, unlike \cite[Theorem 1.6]{DSS}, does not explicitly assume that any time $t$ balls are compactly contained in $M$ or $\Om$. However, we can deduce such a statement from the distance convergence of assumption \eqref{loc_unif_conv2}, as clarified in Remark \ref{tech_rmk}.

The work of Deruelle-Schulze-Simon has prompted the question of whether merely the upper $C/t$ curvature bound is sufficient to deduce the conclusion of smoothness of the flow at time $t=0$ (cf. \cite[Problem 1.1]{DSS} and \cite[Problem 1.8]{DSS}). 
Our first result is that the example of \cite{TY3} already disproves this.
\begin{thm}[Extends \cite{TY3}]
\label{TY_ext_thm}
There exists a  Ricci flow $g(t)$ on $\R^2$, for $t>0$, achieving 
the standard Euclidean metric $g_0$ on $\R^2$ as initial data in the sense of \eqref{loc_unif_conv}, which satisfies the curvature estimate $|K_{g(t)}|\leq \frac{1}{2t}$ for all $t\in (0,T)$,
but which is not the static Ricci flow $g(t)\equiv g_0$.
\end{thm}

The flow $g(t)$ arises by evolving the Radon measure 
$\mu_{g_0}+\h^1\llcorner L$, where $\mu_{g_0}$ is the Lebesgue measure and $L=\{0\}\times \R$ is the $y$-axis in $\R^2$,
using the existence theory of \cite{TY3}.
The novelty compared with \cite{TY3} is that the we are demonstrating that the resulting flow satisfies a curvature bound of the form $|\Rm|_{g(t)}\leq C/t$, even with an explicit constant. Note that the scalar curvature lower bound of B.L. Chen \cite{strong_uniqueness}
tells us that $K\geq -\frac{1}{2t}$ for any complete Ricci flow on any surface. 
We prove in this note the analogous upper bound that 
$$K\leq \frac{1}{2t},$$
not just for the flow above but for any complete flow on $\R^2$ that is invariant under vertical translations.

\begin{thm}
\label{sym_thm}
Suppose that $g(t)=u(x,t)(dx^2+dy^2)$ is any complete Ricci flow on $\R^2$, 
for $t\in (0,T)$,
whose conformal factor is independent of $y$.
Then 
$$|K_{g(t)}|\leq \frac{1}{2t}.$$
\end{thm}
Without the hypothesis that $u$ is independent of $y$, the curvature can even be unbounded for some or all time \cite{GT3, GT4}.
The explicit upper and lower bounds for the curvature given in Theorem \ref{sym_thm} are both sharp. Indeed, both bounds are sharp at every time $t>0$ in the non-gradient expanding Ricci soliton starting with the Radon measure $2\pi\h^1\llcorner L$,
discovered in \cite{TY3}, whose conformal factor is given explicitly as
$$u(x,t)=\frac{2t}{t^2+x^2}.$$
We prove Theorem \ref{sym_thm} in Section \ref{press_sect}.

Returning to Theorem \ref{DSS_thm} of Deruelle-Schulze-Simon, having disproved that condition \eqref{cond_1} can be simply dropped, we ask whether we can drop condition 
\eqref{cond_2} instead. 
We have already noted that a result of this form holds in dimension two, owing to Theorem \ref{TR_thm}.
It turns out that in dimension three this is also true, even locally. 
We will state and prove this assertion in greater generality below, but the idea of proof is as follows:
A uniform lower Ricci bound, coupled with the weak attainment of the initial data, is already enough to rule out volume collapsing of the flow. This is then enough to deduce the $|\Rm|_{g(t)}\leq C/t$ curvature decay over short time periods, and hence reduce to the case of Deruelle-Schulze-Simon. 
In higher dimensions it remains open whether condition \eqref{cond_2} of Theorem \ref{DSS_thm} can be dropped. However, we demonstrate in Section \ref{pos_sect} that the argument above remains valid if we replace the uniform Ricci lower bound hypothesis by a lower bound on a different type of averaging of sectional curvatures
that coincides with a lower bound on the Ricci curvature in three dimensions and is stronger in dimension four and higher. 
More precisely we ask for a uniform lower bound on  all complex sectional curvatures corresponding to degenerate 2-planes, which we write $\K\geq -C$, 
and refer to \cite[Sections 1 and 2]{MT2} for clarification of notation and further details. In the case that 
$\K >0$, 
i.e.  all these averages are strictly positive, the corresponding curvature condition is known as PIC1. This notion originates in the work of Micallef-Moore \cite[\S 5]{MM} and became famous after its use by Brendle and Schoen in the context of the differentiable sphere theorem \cite{BS}.

The result is then the following theorem, in which no metrics are assumed to be complete.
\begin{thm}
\label{IC1_thm}
Suppose $g(t)$ is a smooth (possibly incomplete) Ricci flow on a smooth manifold $M$,
of dimension at least three, for $t\in (0,T)$
that satisfies 
$$\K\geq  -\al_0\quad\text{ on }M\times (0,T),$$
for some $\al_0>0$, and with the property that \eqref{loc_unif_conv2} holds for some $\Om\subset M$. 
Then on $\Om$ we can extend $g(t)$ to a smooth Ricci flow on the time interval 
$[0,T)$.
\end{thm}
\begin{rem}
Although this theorem is stated for dimension at least three,  given a Ricci flow on a surface with Gauss curvature uniformly bounded below such that \eqref{loc_unif_conv2} holds for some open subset $\Om$ of the underlying surface,  
we can take a Cartesian product with $\R$ and apply Theorem \ref{IC1_thm} to deduce that the flow can be extended smoothly to $t=0$ on $\Om$. This gives 
a generalisation of Theorem \ref{TR_thm} to the noncompact case.
\end{rem}

As discussed before Theorem \ref{IC1_thm}, our method of proof will be to reduce to 
Theorem \ref{DSS_thm} of Deruelle-Schulze-Simon.

Time zero regularity results such as Theorem \ref{IC1_thm} have applications beyond the theory of Ricci flow itself. 
We consider now what it says about the preservation of curvature conditions under 
Gromov-Hausdorff convergence. We say a manifold is WPIC1 (weakly PIC1) if $\K\geq 0$ throughout.

\begin{thm}
\label{WPIC1_lim_thm}
For $n\geq 3$, suppose that $(M_i,g_i)$ is a sequence of (not necessarily complete) $n$-dimensional WPIC1 Riemannian manifolds 
such that for all $i$ we have $x_i\in M_i$ and
$B_{g_i}(x_i,1)\subset\subset M_i$.
Suppose that $(B_{g_i}(x_i,1), d_{g_i},x_i)$ converges in the pointed Gromov-Hausdorff sense to an $n$-Riemannian limit $(X,d,p)$.
Then this limit is WPIC1 at $p$.
\end{thm}

Note that because the Gromov-Hausdorff limit in the theorem above is $n$-Riemannian, we have no problem making unambiguous sense of it being WPIC1 at $p$.
The limit is uniquely expressible as a Riemannian manifold, by virtue of the Myers-Steenrod theorem, although its induced Riemannian distance might only agree with the limiting distance metric for points that are near to each other.

There are a number of well-known prior results establishing that non-negativity of certain types of curvature are preserved under Gromov-Hausdorff convergence. 
The theory of Alexandrov spaces deals with the case of nonnegative sectional curvature, while the theory of optimal transportation can handle nonnegative Ricci curvature in arbitrary dimension. Nonnegativity of scalar curvature is not preserved under Gromov-Hausdorff convergence (see also \cite{LT2}).

We will give the proof of Theorem \ref{WPIC1_lim_thm} in Section \ref{GH_WPIC1_sect}.
The strategy we adopt is to construct a local Ricci flow starting with 
%\cmt{DELETE: some neighbourhood of an arbitrary point in} 
the $n$-Riemannian limit as initial data, and to show that this flow almost preserves the $\K$ lower bound. By the time zero regularity of Theorem \ref{IC1_thm} the flow extends smoothly down to time zero, and so the initial data inherits the $\K$ lower bounds of the flow.
Slight variants of this strategy are possible,  all reducing to the regularity theory of Deruelle-Schulze-Simon. 
It is interesting to compare the general strategy with Bamler's proof 
\cite{bamler} of the preservation of scalar curvature lower bounds under $C^0$ convergence of  metric tensors.

\vskip 0.5cm
\noindent
\emph{Acknowledgements:} PT was supported by EPSRC grant EP/T019824/1.

\section{Curvature decay of symmetric Ricci flows on $\R^2$}
\label{press_sect}

In this section we prove Theorem \ref{sym_thm}.

We have already commented that the lower bound $K\geq - \frac{1}{2t}$ is already known.
We prove the upper bound $K\leq \frac{1}{2t}$ using the notion of pressure that is widely used in the study of 
the porous medium equation and the fast diffusion equation.
The calculations we make are close to, for example, those in the work of
Esteban-Rodriguez-Vazquez \cite{ERV}, and the general idea can be traced back to
\cite{AB}. A notable departure from the existing theory of the logarithmic fast diffusion equation is that we do not make any boundedness assumptions on $u$, nor growth or decay conditions at spatial infinity. In Theorem \ref{sym_thm}, all such possible conditions are replaced by the geometric assumption that the flow is complete. The improvements here are possible because we deviate substantially from the earlier literature in our application of the maximum principle, taking a geometric approach and implicitly using the remarkable properties of the distance function and its evolution under Ricci flow.

\begin{proof}[Proof of Theorem \ref{sym_thm}]
On a surface, the Ricci flow can be written $\pl{g}{t}=-2Kg$, where 
$K$ is the Gauss curvature. If we write the metric locally as 
$u(dx^2+dy^2)$ then $K=-\frac{1}{2u}\lap \log u$, where $\lap$ is the Laplacian with respect to the isothermal coordinates $x$ and $y$, and so the conformal factor $u$ evolves under the equation $u_t=\lap \log u$.
%Under the Ricci flow, the conformal factor $u$ of a metric $u(dx^2+dy^2)$ evolves under the equation $u_t=\lap \log u$.
Because of the symmetry of the flow in this theorem, the conformal factor can be viewed as a 
solution $u\colon \R\times (0,T)\to (0,\infty)$ to the equation
$$u_t=(\log u)_{xx}.$$
A short calculation confirms that $w:=\frac{1}{u}$, which coincides with the so-called pressure up to a sign, solves the equation
$$w_t=w.w_{xx}-(w_x)^2.$$ 
The quantity $q:=w_{xx}$ then satisfies the equation
\beqa
q_t 
&= [w_t]_{xx}\\
&=[wq-(w_x)^2]_{xx}\\
&=w.q_{xx}+ 2w_x q_x+ q^2-(2q_xw_x+2q^2)\\
&= w.q_{xx}-q^2.
\eeqa
Forgetting the symmetry of the flow, and viewing $q$ as a function on the whole of 
$\R^2\times (0,T)$ that is independent of $y$, this amounts to
$$q_t = \lap_g q - q^2,$$
where $\lap_g$ is the Laplace-Beltrami operator associated with $g(t)$.

Formally, one might then try to make a comparison with 
the function $t\mapsto 1/t$, which solves the ODE $q_t=-q^2$, 
to deduce that $q\leq 1/t$.
But one must worry about the validity of such a comparison when so little is known about the behaviour of $q$ at spatial infinity. 
To avoid this problem we must appeal to the special properties of the Riemannian distance function under Ricci flow that were discovered by Perelman \cite{P1}. He showed how to use these to construct useful time-dependent cut-off functions, and B.-L. Chen \cite{strong_uniqueness} 
invoked this idea when proving his scalar curvature estimates that were alluded to earlier. The key idea of Chen's argument was distilled into the following result by 
Huang, Tam, Tong and the first author.
\begin{lem}[{Special case of \cite[Lemma 5.1]{HLTT}}]
Suppose that $g(t)$ is a complete Ricci flow on a manifold $M$ of general dimension, for $t\in (0,T)$. Suppose that $q:M\times (0,T)\to\R$ is a smooth function satisfying
$$q_t \leq \lap_g q - \al q^2,$$
for some $\al>0$. Then 
$$q(x,t)\leq \frac{1}{\al t}$$
throughout $M\times (0,T)$.
\end{lem}

Applying this lemma in our situation then implies the desired bound 
$q(x,t)\leq \frac{1}{t}$.
But a calculation shows that 
\beqa
2K & =-\frac{1}{u}|(\log u)_x|^2+q\\
&\leq q\leq \frac{1}{t},
\eeqa
which concludes the proof of the theorem.
\end{proof}

\begin{rem}
\label{r22}
There is an alternative proof available of Theorem \ref{sym_thm}, albeit with a non-explicit constant in the curvature estimate, using a more geometric approach.
It turns out that any complete Riemannian metric of the form
$u(x)(dx^2+dy^2)$, i.e. with conformal factor independent of $y$, enjoys a positive lower bound on its volume ratio at all scales. This alone is enough to obtain $C/t$ decay of curvature using an approach similar to that used in the proof of \cite[Lemma 2.1]{ST1}.
%discussed in Section \ref{pos_sect}.
\end{rem}

\section{Time zero regularity in the presence of an \I{} lower bound}
\label{pos_sect}

In this section we prove Theorem \ref{IC1_thm} by reducing the problem to the case covered by Theorem \ref{DSS_thm} of Deruelle-Schulze-Simon.
Neither Theorem \ref{IC1_thm} nor Theorem \ref{DSS_thm} explicitly assumes that any time $t$ balls are compactly contained in $M$ or $\Om$. However, we can deduce such a statement from the distance convergence of assumption \eqref{loc_unif_conv2} as we explain in the following remark. 

\begin{rem} %[Technical remark.]
\label{tech_rmk}
With a little care, assumption \eqref{loc_unif_conv2} allows us to compare $d_0$ balls to $g(t)$ balls.
Suppose
$x_0\in \Om$, and $r>0$ is sufficiently small so that 
$B_{d_0}(x_0,2r)\subset\subset\Om$.
For sufficiently small $t>0$, $d_{g(t)}(\cdot,x_0)$ and $d_0(\cdot,x_0)$ can only differ by at most $r$ on $\overline{B_{d_0}(x_0,2r)}$, by \eqref{loc_unif_conv2}, 
and we claim that this implies
$$B_{g(t)}(x_0,r)\subset B_{d_0}(x_0,2r).$$ 
This needs some care because
a point in $B_{g(t)}(x_0,r)$ does not a priori live in the region where
$d_{g(t)}(\cdot,x_0)$ and $d_0(\cdot,x_0)$ are close. 
However, if we could pick $\ti y\in B_{g(t)}(x_0,r)$
that did not lie in $B_{d_0}(x_0,2r)$ then by moving along a path 
from $x_0$ to $\ti y$ of $g(t)$-length less than $r$, we would at some point get to a point
$y\in B_{g(t)}(x_0,r)$ with $d_0(x_0,y)=2r$. 
But by our assumption on $t$, this would be a contradiction.

Observe  also that the distance convergence of assumption \eqref{loc_unif_conv2} immediately implies that 
$$B_{d_0}(x_0,r/2) \subset B_{g(t)}(x_0,r),$$
for sufficiently small $t$, so we have inclusions in both directions.
\end{rem}

\begin{proof}[Proof of Theorem \ref{IC1_thm}]
As previously mentioned, the hypothesis $\K\geq -\al_0$ implies a lower Ricci bound
$\Ric\geq -C$, where $C$ depends on $\al_0$ and the dimension of $M$.
In order to apply Theorem \ref{DSS_thm} we are therefore lacking only a bound of the form
$|\Rm|_{g(t)}\leq C/t$.

It suffices to obtain the smooth extension in a neighbourhood of an arbitrary point 
$x_0\in \Om$. Pick $r>0$ sufficiently small so that $B_{d_0}(x_0,2r)\subset\subset \Om$
and also so that $B_{d_0}(x_0,2r)$ is isometric to an $n$-dimensional Riemannian manifold $(N,h)$ (using the hypothesis \eqref{loc_unif_conv2}) with $x_0$ corresponding to $y\in N$.

By Remark \ref{tech_rmk}, after possibly reducing $T>0$
we may assume that 
$$B_{d_0}(x_0,r/2) \subset B_{g(t)}(x_0,r)\subset B_{d_0}(x_0,2r)\subset\subset\Om$$ 
for every $t\in (0,T]$.
By the uniform lower Ricci bound, and the fact that
\begin{equation*}
%\label{vol_evol}
\frac{d}{dt} \Vol_{g(t)}(B_{d_0}(x_0,\frac{r}{2}))
=-\int_{B_{d_0}(x_0,\frac{r}{2})}R\,dV_{g(t)},
\end{equation*}
(see e.g. \cite[(2.5.7]{RFnotes})
the volume $\Vol_{g(t)}(B_{d_0}(x_0,r/2))$ can only decay exponentially as $t\downto 0$, and in particular, it has a positive lower bound $v_0$ independent of how small we take $t>0$. In particular, we find that 
$$\VolB_{g(t)}(x_0,r)\geq v_0>0$$
for all $t\in (0,T]$.

In the case that $M$ is three-dimensional, we can now apply Lemma 2.1 from \cite{ST1}
to the Ricci flows $t\mapsto g(t+\ep)$, for arbitrarily small $\ep>0$, in order to deduce that $|\Rm|_{g(t)}\leq C/t$ on $B_{g(t)}(x_0,r/2)$ over some small time interval 
$(0,\hat T)$ as required. Yi Lai adapted Lemma 2.1 from \cite{ST1} to higher dimensions assuming a lower bound $\K\geq -C$ as we have here, using theory of 
Bamler, Cabezas-Rivas and Wilking \cite[Lemma 4.2]{BCRW}; see \cite[Lemma 3.4]{yi_lai}.

Either way, we have obtained the required curvature decay 
on $B_{g(t)}(x_0,r/2)$, and hence on $B_{d_0}(x_0,r/4)$ by Remark \ref{tech_rmk}
(after possibly shortening the time interval)
and we can invoke Theorem \ref{DSS_thm} to establish the required smooth extension to $t=0$ near $x_0$.
\end{proof}

\section{Gromov-Hausdorff limits of WPIC1 manifolds are WPIC1}
\label{GH_WPIC1_sect}

In this section we prove Theorem \ref{WPIC1_lim_thm}.
In the process, we will need to run the  Ricci flow locally starting with 
(subsets of) the approximations  $B_{g_i}(x_i,1)$ while retaining a lower uniform curvature bound and gaining $C/t$ decay of the curvature tensor. 
The existence of such a flow was proved in three dimensions in \cite{ST2}, and could also be derived from a combination of
\cite{Hochard} and \cite{ST1}. The proof was extended by Lai \cite{yi_lai}
and Hochard \cite{hochard_thesis} to the analogous result in higher dimensions, and the following version of their extensions is a special case of 
\cite[Theorem 3.3 ]{MT2}.
\begin{thm}
\label{loc_exist_thm}
Given $n\geq 3$ and $\al_0,v_0>0$, there exist $C, T>0$ such that if 
$(M,g_0, x_0)$ is a smooth pointed Riemannian $n$-manifold with $B_{g_0}(x_0,2)\subset\subset M$,  $\K({g_0})\geq -\al_0$,  and $\VolB_{g_0}(x_0,1)\geq v_0$, 
then there exists a smooth Ricci flow $g(t)$, $t\in [0,T]$, on $B_{g_0}(x_0,1)$
with $g(0)=g_0$ where defined, such that for all $t\in (0,T]$ we have
$$|\Rm|_{g(t)}\leq \frac{C}{t}\quad\text{and}\quad
\K({g(t)})\geq -C$$
throughout $B_{g_0}(x_0,1)$, and $\inj_{g(t)}(x_0)\geq \sqrt{\frac{t}{C}}$.
\end{thm}

\begin{proof}[Proof of Theorem \ref{WPIC1_lim_thm}]
By hypothesis, a neighbourhood $U$ of $p$ within the limit is isometric to an $n$-dimensional Riemannian manifold $(N,h)$. For one such isometry $\vph:U\to N$, define
$y=\vph(p)$. By taking $r\in (0,\half)$  sufficiently small we can be sure both that 
$B_h(y,2r)\subset\subset N$, and that for all $z_1,z_2\in B_h(y,r)$
the distance $d_h(z_1,z_2)$ is realised by a geodesic lying entirely within 
$B_h(y,r)$.
Then $B_{g_i}(x_i,r)\subset\subset M_i$
and we have pointed Gromov-Hausdorff convergence 
$$(B_{g_i}(x_i,r),d_{g_i},x_i)\to (B_h(y,r),d_h,y).$$ 
Our objective is to prove that $\K\geq 0$ at $y$.
By rescaling, we may as well assume that $r=2$.

%Suppose we would like to know that $\K\geq 0$ at an arbitrary point $\ti y$ in the limit.
%
%By hypothesis, a neighbourhood $U$ of $\ti y$ within the limit is isometric to an $n$-dimensional Riemannian manifold $(N,h)$. For one such isometry $\vph:U\to N$, define
%$y=\vph(\ti y)$. By taking $r>0$  sufficiently small we can be sure both that 
%$B_h(y,2r)\subset\subset N$, and that for all $z_1,z_2\in B_h(y,r)$
%the distance $d_h(z_1,z_2)$ is realised by a geodesic lying entirely within 
%$B_h(y,r)$.
%
%%\cmt{We omit finitely many terms below because initially we might have $B_{g_i}(x_i,r)$ hanging slightly off the edge of $M_i$}
%
%By adjusting each $x_i$ to some other points in the original balls $B_{g_i}(x_i,1)$
%(and possibly omitting finitely many terms) 
%we can then assume that we have balls $B_{g_i}(x_i,r)\subset\subset M_i$
%and pointed Gromov-Hausdorff convergence 
%$(B_{g_i}(x_i,r),d_{g_i},x_i)\to (B_h(y,r),d_h,y)$, and 
%reduce to proving that $\K\geq 0$ at $y$.
%In fact, by rescaling, we may as well assume that $r=2$.

Although Theorem \ref{WPIC1_lim_thm} does not refer explicitly to volume bounds, each 
$(B_{g_i}(x_i,2),g_i)$ has nonnegative Ricci curvature (implied by WPIC1 in all dimensions $n\geq 3$) and so Colding's volume convergence \cite{colding}
implies that $\VolB_{g_i}(x_i,1)\to \VolB_{h}(y,1)$ as $i\to\infty$. 
In particular, defining $v_0=\half \VolB_{h}(y,1)>0$ we may assume (after dropping finitely many terms in $i$) that 
$$\VolB_{g_i}(x_i,1)\geq v_0>0$$
for all $i$.
Note that Colding's work is phrased in the weaker context of complete manifolds with uniform Ricci lower bounds, but the same techniques can be extended to this more general local situation.

We can apply Theorem \ref{loc_exist_thm} to each $(M_i,g_i,x_i)$ with $\al_0=1$ and $v_0$ as above, to obtain a sequence of Ricci flows $g_i(t)$ on $B_{g_i}(x_i,1)$
with $g_i(0)=g_i$ where defined, on a common time interval $[0,T]$, such that for all $t\in (0,T]$ we have
$$|\Rm|_{g_i(t)}\leq \frac{C}{t}\quad\text{and}\quad
\K({g_i(t)})\geq -C$$
throughout $B_{g_i}(x_i,1)$, and $\inj_{g_i(t)}(x_i)\geq \sqrt{\frac{t}{C}}$,
for some $i$-independent $C$.
The estimates satisfied by these flows then give local compactness  
similarly to as in \cite{ST2, MT1} with full details given in Section 3  of \cite{mcleod_thesis} (see e.g. \cite[Theorem 3.6.1]{mcleod_thesis}).
In particular, after passing to a subsequence and possibly reducing $T>0$, 
there exist an $n$-dimensional smooth manifold $M$ and a  smooth Ricci flow
$g(t)$ on $M$ for $t\in (0,T]$, a point $x_0\in M$ and
a sequence of smooth maps $\vph_i:M\to M_i$ that are diffeomorphic onto their images
with $\vph_i(x_0)=x_i$, such that 
$$\vph_i^*(g_i(t))\to g(t)$$
smoothly locally on $M\times (0,T]$ as $i\to\infty$.
Moreover, by the control on distances given in \cite[Lemma 3.1]{ST2} we have
$d_{g(t)}\to d_0$ uniformly as $t\downto 0$, where $d_0$ is a distance metric such that $(M,d_0)$ is isometric to a neighbourhood of $y$ in $(N,h)$, with the isometry sending $x_0$ to $y$. 
Note that we are happy here if the limit Ricci flow $(M,g(t))$ corresponds to 
only a tiny neighbourhood of $x_i$ in each approximating flow.
Indeed, this then allows us to assume that $M$ is a ball in $\R^n$ and to prove the compactness using only one harmonic coordinate chart for each $i$, although we do not need this simplification. 

Because the convergence of the flows is smooth, the limit Ricci flow inherits the curvature bound $\K({g(t)})\geq -C$ and we can apply Theorem \ref{IC1_thm}
in order to deduce that we can smoothly extend $g(t)$ to the whole time interval 
$[0,T]$ by making a suitable definition of $g(0)$.
Neighbourhoods of $x_0$ and $y$ in the Riemannian manifolds $(M,g(0))$ and $(N,h)$, respectively, are then isometric as metric spaces and thus as Riemannian manifolds by the Myers-Steenrod theorem.
Our task then reduces to proving that $\K(g(0))\geq 0$ at $x_0$. By smoothness, we have so far established that $\K(g(0))\geq -C$ at $x_0$.

To show that $g(0)$ is WPIC1 at $x_0$,
we need to return to the Ricci flows $g_i(t)$ and obtain better lower curvature bounds that become uniformly closer to WPIC1 as $t\downto 0$.

To achieve this we will appeal to Theorem 3.1 of \cite{LeeTam}. That theorem, applied to each $g_i(t)$ on its domain $B_{g_i}(x_i,1)$, tells us that because the curvature of each of the flows $g_i(t)$ decays like $C/t$ 
and the initial data is WPIC1 at $t=0$ for each $i$,
then after possibly reducing $T>0$ (independently of $i$) we can estimate 
$$\K(g_i(t))\geq -c_1t^l$$
at $x_i$ for some $c_1<\infty$ and $l>1$ dependent on the curvature bound $C$.
In other words, although we cannot hope for the property $\K\geq 0$ to be preserved for a local Ricci flow without boundary conditions, we retain a lower bound for $\K$ that can be made as close as we like to zero as $t\downto 0$. Moreover, this lower bound is independent of $i$ and so can be passed to the limit $i\to\infty$ to give
$$\K(g(t))\geq -c_1t^l$$
at $x_0$.
Because we established the time zero regularity of $g(t)$, this lower bound can now be passed to the limit $t\downto 0$ to conclude that $(M,g(0))$ is WPIC1 at $x_0$, as
required.
\end{proof}

\vskip 0.5cm
\noindent
%Man-Chun Lee: 
\href{mailto:mclee@math.cuhk.edu.hk}{mclee@math.cuhk.edu.hk}

\noindent
{{\sc {Department of Mathematics, The Chinese University of Hong Kong, Shatin, N.T., Hong Kong}}

\vskip 0.2cm

\noindent
\url{https://homepages.warwick.ac.uk/~maseq/}
%\url{http://warwick.ac.uk/fac/sci/maths/people/staff/peter\_topping/} 

\noindent
{\sc Mathematics Institute, University of Warwick, Coventry,
CV4 7AL, UK.}

\end{document}